\documentclass{amsart}

\RequirePackage{graphicx}
\RequirePackage{caption}

\usepackage[utf8]{inputenc}
\usepackage[T1]{fontenc} 
\usepackage{hyperref}   
\usepackage{url}         
\usepackage{booktabs}  
\usepackage{amsfonts}      
\usepackage{nicefrac}       
\usepackage{microtype}      
\usepackage{xcolor}     
\usepackage{amsthm}

\usepackage{amsmath}
\usepackage{amsfonts}
\usepackage{amssymb}
\usepackage{cuted}
\usepackage{enumitem}
\usepackage{url}

\usepackage[square,sort,comma,numbers]{natbib}

\usepackage{afterpage}

\usepackage{algorithm}
\usepackage{algorithmic}

\newcommand\blankpage{%
    \null
    \thispagestyle{empty}%
    \addtocounter{page}{-1}%
    \newpage}

\newcommand{\pr}{\mbox{\sf P}}
\newcommand{\ex}{{\bf\sf E}}               
\newcommand{\var}{\mbox{\sf Var}}

\newcommand{\bZ}{{\mathbb Z}}
\newcommand{\bx}{{\bf x}}               

\newcommand{\by}{{\bf y}}               

\newcommand{\calP}{{\mathcal P}}

\newcommand{\calS}{{\mathcal S}}
\newcommand{\calI}{{\mathcal I}}

\newcommand{\calE}{{\mathcal E}}

\newcommand{\calF}{{\mathcal F}}

\newcommand{\ra}{\rightarrow}           

\def\ex{\mathbb{E}}

\newcommand{\tvert}{|\!|\!|}
\newcommand{\Real}{\mathbb R}

\newcommand{\bbP}{\mathbb P}
\newcommand{\bbQ}{\mathbb Q}

\newcommand{\lP}{L_\bbP}
\newcommand{\lQ}{L_\bbQ}
\newcommand{\lQP}{L_{\bbQ\times\bbP}}
\newcommand{\Sample}{\Sigma}
\newcommand{\Proj}{\Pi}
\newcommand{\Hamil}{\mathcal H}
\newcommand{\Hamiltonian}{\mathbf H}

\newcommand{\bbg}{\mathfrak{g}}
\newcommand{\bbf}{\mathfrak{f}}

\newcommand{\bbU}{\mathbf U}
\newcommand{\bbV}{\mathbf V}

\newcommand{\g}{\lambda}                
\newcommand{\Lam}{\Lambda}               

\newtheorem{defn}{Definition}
\newtheorem{pro}{Proposition}[section]
\newtheorem{thm}{Theorem}[section]
\newtheorem{lem}{Lemma}[section]
\newtheorem{rem}{Remark}[section]
\newtheorem{asm}{Assumption}
\newtheorem{cor}{Corollary}
\newtheorem{clm}{Claim}

\newcommand{\bbt}{{2}}
\newcommand{\Lt}{{L^\bbt}}

\newcommand\distribution{\mathfrak{f}}
\newcommand{\bbp}{\mathfrak{p}}
\newcommand{\bbq}{\mathfrak{q}}

\newcommand{\supp}{\rm supp\,}
\newcommand{\suppQ}{\mathfrak{Q}}
\newcommand{\suppP}{\mathfrak{P}}

\newcommand{\cP}{\mathcal{P}}

\newcommand{\cT}{\mathcal{T}}

\newcommand{\YL}[1]{{\color{blue}YL #1}}

\usepackage[normalem]{ulem}

\title{On Convergence of the Alternating Directions SGHMC Algorithm}
\usepackage{times}

\author{%
  Soumyadip Ghosh, \quad Yingdong Lu, \quad Tomasz Nowicki \\
  Department of Mathematical Sciences\\
  IBM Research\\
  Yorktown Heights, NY 10598\\
}

\begin{document}

\maketitle

\begin{abstract}%
We study convergence rates of Hamiltonian Monte Carlo (HMC) algorithms with leapfrog integration under mild conditions on stochastic gradient oracle for the target distribution (SGHMC). Our method extends standard HMC by allowing the use of 
general auxiliary distributions, which is achieved by a novel procedure of Alternating Directions.

The convergence analysis is based on the investigations of the Dirichlet forms associated with the underlying Markov chain driving the algorithms. For this purpose, we provide a detailed analysis on the error of the leapfrog integrator for Hamiltonian motions with both the kinetic and potential energy functions in general form.  
We characterize the explicit dependence of the convergence rates on key parameters such as the problem dimension, functional properties of both the target and auxiliary distributions, and the quality of the oracle.  
\end{abstract}


\section{Introduction}
\label{sec:intro}


The Hamiltonian Monte Carlo (HMC) algorithm has its humble beginning in physics, and in the recent literature it has seen much wider application in modern statistical analysis (inference and learning) and artificial intelligence. this has also subsequently generated a much deeper understanding of the method. 
Given a \emph{target} distribution which is known to be proportional to a given positive, integrable function $\bbf$, Markov chain Monte Carlo (MCMC) algorithms are commonly employed to provide either estimations of the normalizing constant (aka partition function) to $\bbf$ or samples from the target distribution. HMC, a member of the MCMC family, utilizes the invariance and ergodic properties of Hamiltonian motion to demonstrate additional benefits in performance against generic MCMC algorithms. With a help of a user chosen \emph{auxiliary} distribution $\bbg$ the algorithm generates a (Hamiltonian) motion while preserving the energy $\Hamil(q,p)=U(q)+V(p)$, where $U(q)= -\log \bbf(q)$ represents the potential energy and $V(p)=-\log\bbg(p)$ the kinetic energy. The chief advantage lies in how well the motion dynamics are implemented; an exact implementation preserves the (joint) density even when making large moves and hence does not require a Metropolis-Hastings (MH) style rejection step to ensure consistency. In practical instances where the motion is approximate, excellent discrete motion implementations such as the Leafprog symplectic integrator ensure that rejection probability is low in the necessary MH correction even in high dimensional settings. 

The literature focuses on analysis of HMC with Gaussians as auxiliary distribution, which corresponds to a rather simple quadratic kinetic energy function. There are several different approaches for both qualitative and quantitative analysis on the key question of the convergence and performance of the HMC algorithms with Gaussian auxiliaries, see,e.g.~\cite{ZouGu2021, 10.1007/s10994-019-05825-y,10.1287/opre.2021.2162}. 


{\bf General auxiliary distributions.} As HMC is increasingly used in more complex applications, for example large language models, it becomes desirable that the algorithm, as well as its analysis, can be extended to allow more flexibility in the selection of the auxiliary distributions. This requires a certain novel algorithm called the Alternating Direction HMC~\cite{GhoshLuNowicki+2022}; we describe it in the sequel. As demonstrated in~\cite{GhoshLuNowicki+2022}, the careful selection of non-Gaussian auxiliary forms can in fact considerably improve the performance of the HMC algorithm. 

{\bf Stochastic implementation of gradient calculation.} In many practical situations, the gradient of the potential energy $U(q)$ of the target density function, which is essential for the running of HMC, is not available or difficult to compute. One approach in this case is to substitute the exact calculation of the function $\nabla U( q)$ by an (unbiased) estimator ${\tilde G_r}( q,\xi)$ with an independent random variable $\xi$. Generally speaking, the computation complexity for calculating ${\tilde G_r}( q,\xi)$ is considered to be significantly less than that of $\nabla U( q)$, especially in high dimensions. Examples of such estimators include Mini-batch stochastic gradient, Stochastic variance reduced gradient, Stochastic averaged gradient and Control variate gradient, as summarized in~\cite{ZouGu2021}.


\textbf{Convergence Analysis of HMC.}
Our primary purpose is to present a unified quantitative convergence analysis of the family of Alternating Direction HMC algorithms that allow for \emph{arbitrary} not necessarily symmetric auxiliary distributions. The analysis presented here also allows for stochastic (inexpensive) oracles that estimate the gradient $\nabla U(q)$ of the potential of the target distribution, and where Hamiltonian motion is implemented with leapfrog symplectic integrators (requiring additional MH correction). We take an analytical approach based on the analysis of the Dirichlet forms that are defined by the underlying Markov chain, and are able to quantitatively characterize convergence rates of Alternating Direction SGHMC under mild conditions on the target density, the (arbitrary) auxiliary density and the stochastic gradient oracle implementations. These, to the best of our knowledge, are the first set of results on the convergence rates of HMC with this kind of generality. 

As a special instance of MCMC, an HMC algorithm is driven by a Markov chain in a general state space. Hence, the analysis of its convergence relies on the analysis of this Markov chain. Convergence of Markov chains, or more general Markov processes, is a central topic in probability theory. The variety of different approaches is larger than what a few books, see, e.g.~\cite{meyn93} and~\cite{levinmarkov}, can cover. The main approach taken in this paper is based on the Dirichlet form, a systematic treatment of which can be found in~\cite{FukushimaOshimaTakeda+2010}. Intuitively, this analysis focuses on establishing functional relationships that quantitatively characterize the evolution of the Markov chain, thus facilitating the convergence analysis. 

In an abstract and general sense, a Dirichlet form is a non-negative definite symmetric bilinear form defined on a Hilbert space, which furthermore is both Markovian and closed. For each Markov chain, a specific Dirichlet form can be defined naturally through a Markov operator on the Hilbert space, which defines this Markov chain. Moreover, due to the celebrated result of Jeff Cheeger, a quantitative relationship between the Dirichlet form and the variance term characterizes the spectral gap of the Markov chain, which in turn is directly related to the convergence rate. These concepts will be introduced in Sec. \ref{sec:algorithm+assumptions}. The main technical portion of the paper will address the estimation of the Dirichlet form.

\textbf{Literature review.}
The research on the convergence and convergence rates of HMC has been concentrated on the case of the auxiliary distribution $\bbg(p)$ being a (conditional) Gaussian. Theoretical understanding of geometric convergence have been developed for these cases, via analytical methods (including comparison theorems for differential equations in~\cite{chenvempala2019}), or probabilistic methods, such as Harris recurrence techniques~\cite{bou-rabee2017} and coupling~\cite{bou-rabee2020}.  For HMC with general auxiliary distributions, qualitative results are obtained in~\cite{HMC_L2} and~\cite{HMC_Lp}. 

The problem we are dealing with here appears to be one instance of the general problem of HMC with a stochastic gradient. While there are existing results for its convergence, see e.g.\cite{ZouGu2021, 10.1007/s10994-019-05825-y,10.1287/opre.2021.2162}, our analysis presented in this paper with explicit convergence rate estimates under very general assumptions (general auxiliary distribution, stochastic gradient implementation and alternating direction) are certainly innovative results. 

Meanwhile, there are also extensive quantitative studies on establishing the dependence of convergence rates on parameters of the algorithms, including, the dimension of the underlying space, the function properties of the target distribution and the quality of the numerical integrators. For performance of the unadjusted HMC (HMC with numerical integrators but without a Metropolis-Hastings step), in Wasserstein distances, see e.g.~\cite{Gouraud2023},~\cite{ShenLee2019},~\cite{CaoLuWang2021},~\cite{Bou-RabeeMarsden2022}. Similarly, results in~\cite{10.5555/3327345.3327502}, ~\cite{ChenDwivediWainwrightYu2020},~\cite{BeskosPillaiRobertsplus2013}, and~\cite{ChenGatmiry} quantifies "gradient complexity"(the amount of gradient calculation required) for HMC with Metropolis-Hastings adjustment. 



\textbf{Summary of our contributions}

{\bf Error estimations:} We present a detailed and comprehensive analysis in Lemmata~\ref{lem:error_general} and~\ref{lem:error_general_SG} on the quality of Leapfrog implementations of the symplectic integration for Hamiltonian equations with general kinetic energy. This is the key for the analysis of HMC with general auxiliary distributions and stochastic gradient. It not only serves as the main technical component for convergence results in this paper, but it can also be used as a building block for the analysis of many other variations of HMC, such as AD-HMC seen in this paper, as well as other systems where sympletic integrations and estimation are required. 

{\bf Convergence:} Quantitative bound on the performance for SGHMC algorithms with general auxiliary distributions are derived. To our best knowledge, this is the first such results with such kind of generality. In addition, these bounds are expressed in explicit forms of the system parameters including the dimension.

{\bf Methods: }The method we used here consists of Dirichlet form and functional inequalities. They offer clearness in concepts and flexibility in analysis, and appear to be promising in achieving both qualitative and quantitative results, and we hope that they would find more applications within this community. We also aim to remove some of the restrictions and apply them to more general systems in the future.

\textbf{Organization of the paper}
%
The rest of the paper is organized as follow: in Sec.~\ref{sec:algorithm+assumptions}, we introduce the HMC algorithm and provide details on its various implementation, and list the assumptions on the functions; geometric convergence is discussed in Sec.~\ref{sec:geometric_convergence}; some ramifications will be presented in Sec.~\ref{sec:ramifications}; and the paper concludes in Sec.~\ref{sec:conclusions}.

\section{Algorithms and Assumptions}
\label{sec:algorithm+assumptions}

\subsection*{Definitions, Notations and Assumptions} 
For any $\bbq\in \Real^d$, and $p\in \bZ_+$, the $p$-norm is defined as $\|q\|_p=(\sum_{i=1}^d q_i^p)^{1/p}$.
\\
For a random variable defined on $\Real^d$, this \emph{can be extended to} 
$\displaystyle{|||q|||_p:=(\ex\|q\|_2^p)^{1/p}}$.
\\
For a $d\times d$ matrix $A$, the operator norm (aka spectral norm) is defined as $\|A\|=\sup_{\|x\|_2=1} \|Ax\|_2$ and Frobenius norm as $\|A\|_F=\sqrt{\sum_{i,j=1}^d A_{ij}^2}$. For any function $f$, $\nabla^3 f$ can be treated as a tensor, and 
\begin{align*}
\|\nabla^3 f\| =\sup\left\{\bigg|\sum_{i,j,k=1}^k \frac{\partial ^3 f}{\partial x_i\partial x_j\partial x_k} u_i v_j w_k\bigg|:\|u\|_2, \|v\|_2, \|w\|_2\le 1\right\}.
\end{align*}

One of the key assumptions in~\cite{GhoshLuNowicki+2022} under which the \emph{geometric convergence} of HMC is established is the \emph{uniform strongly logarithmic concavity} of both the target and auxiliary distributions. This is equivalent to making an assumption on the convexity and derivative-Lipshitzness conditions on the energy functions. 

\begin{defn}[{\bf $\calS_{\ell, L}(\Real^d)$ class}]\label{defn:l-convex-fn}
A function $W:\Real^d \ra \Real$ is called to be a class  
$\calS_{\ell, L}$ for some $\ell, L> 0$ if the following holds for any $x_1, x_2\in \Real^d$ and $t\in[0,1]$, 
\begin{align*}
W((1-t)x_1+t x_2) \;\le\; & (1-t)W(x_1) +t W(x_2)-\frac{\ell}{2}t(1-t)\|x_1-x_2\|^2,
\end{align*}
and $\|\nabla W(x_2)-\nabla W(x_1) \|\le L \|x_2-x_1\|$.
\end{defn}
\begin{rem}
The function class $\calS_{\ell, L}$ is the same as that of $\calS^{1,1}_{\ell, L}(\Real^d)$ class in~\cite{nesterov2003introductory}. 
It should be easy to see that 
$\ell\, {\rm Id}\preceq \nabla^2 W \preceq L\, {\rm Id} $, if $\nabla^2 W$ exists. For any two matrices $A$ and $B$, $A \preceq B$ means that $B-A$ is positive semidefinite. 
\end{rem}
\begin{asm}
There exist $0<\ell_U\le L_U<\infty$ and $0<\ell_V\le L_V<\infty$ such that, 
\label{asm:matrixspace}
$U \in \calS_{\ell_U, L_U}(\Real^d)$, and $V \in 
\calS_{\ell_V, L_V}(\Real^d)$.
\end{asm}
\begin{asm}
Both $U$ and $V$ have third derivatives, and there exist $0<T_U, T_V<\infty$ such that $\sup_{q\in \Real^d}\|\nabla^3 U(q)\|\le T_U$ and $\sup_{q\in \Real^d}\|\nabla^3 V(q)\|\le T_V$.
\label{asm:third_derivative}
\end{asm}

\subsection*{Dirichlet Form and Spectral Gap.} 

Dirichlet form, as a generalization of the Laplace operator, is an important concept in analysis, a systematic treatment of its connection to probability theory, especially the  symmetric Markov processes can be found in~\cite{FukushimaOshimaTakeda+2010}.
\begin{defn}
A symmetric bilinear form $\calE(\cdot, \cdot)$ on the Hilbert space $L^2(X,m)$ with $X$ being a metric space and $m$ a Borel measure is \emph{Markovian} if for any $\epsilon>0$, there exists a real function $\phi_\epsilon(t):\Real \ra \Real$ satisfying 
$\phi_\epsilon(t) = t$ for $t\in[0,1]$,  $ \phi_\epsilon(t) \in [-\epsilon, 1+\epsilon]$, and $0\le \phi_\epsilon(t')-\phi_\epsilon(t)\le t'-t$ whenever $t<t'$, such that $\calE(\phi_\epsilon(u), \phi_\epsilon(u))\le \calE(u,u)$.
\end{defn}
\begin{defn}
A symmetric bilinear form is a \emph{Dirichlet form} if it is both Markovian and closed.
\end{defn}

For a reversible Markov chain on $\Real^d$ with invariant measure $\pi(x)$ and transition kernel $P(x,A)$,  , such as the ones we considered here in this paper, the following gives a natural Dirichlet form on $L^2(\Real^d, \pi)$ (without causing confusion, we will write $L^2$ in the sequel), 
\begin{align*}
\calE(g,h) = \int_X\int_X [g(x)-g(y)][h(x)-h(y)]\pi(dx)P(x, dy).
\end{align*}
Moreover, the spectral gap of such Markov chain, $1-\g_2$, has the following representation,
\begin{align*}
1-\g_2= \inf_{\hbox{\quad $h$ not constant}} \frac{\calE(h,h)}{\var_\pi(h)},
\end{align*}
with $\g_2$
represents the second largest eigenvalue and $\var_\pi(h):= \int_X\int_X (h(x)-h(y))^2\pi(dx)\pi(dy)$. The Dirichlet form approach on the convergence rate is closely related the study of conductance that originated by Jeff Cheeger~\cite{cheeger1969} and carried out by a series of subsequent studies. A detailed exposition of the results and basic arguments can be found in~\cite{46b1ae2c-bf12-35e3-bc28-6851efd139ec}. For Markov chain generated by the HMC algorithm with leapfrog implementation, the presence of invariant measure(up to a constant) and explicit form of transition make the Dirichlet form approach very appealing.

\subsection*{Hamiltonian Monte-Carlo Algorithms}

\textbf{Basic Algorithms.} 
A generic HMC algorithm, see Algorithm~\ref{algo:hmc}, on Euclidean space usually consists of three operations at each step, with a starting point $q\in \Real^d$ is given: (1) ``lift'', it is also called ``spread '' in literature, where a sample $p$ is drawn from the auxiliary distribution with density $\bbg(\cdot)$,  $(q,p)$ will be the point in the (symplectic) space of $\Real^{2d}$; 2) ``rotation''  to a new point, $({\hat q},{\hat p})$, is identified by the Hamiltonian trajectory with energy $\Hamil(q,p)=-\log[\bbf(q)\bbg(p)]=U(q)+V(p)$,  represented here by its potential and kinetic components; 3) ``projection'', ${\hat q}$ will be the starting point of the next step.

\textbf{Leapfrog implementation of the symplectic integration.}
The exact Hamiltonian integration to get from $(q,p)$ to $({\hat q},{\hat p})$ is in general expensive to calculate, and the following well-known Leapfrog approximation is considered here.
\begin{align}
{\hat q} =& q+\eta \nabla V\left(p -\frac12 \eta \nabla U(q)\right), \quad \quad
{\hat p} = p- \eta\frac{ \nabla U(q) + \nabla U({\hat q})}{2}. \label{eqn:leapfrog_basic}
\end{align}

\textbf{An Acceptance/Rejection step.}
At each step of the HMC implementation, leapfrog procedure~\ref{eqn:leapfrog_basic} will be invoked $K$ time, produce a proposal $(q_K, p_K)$ from the initial state $(q_0, p_0)$ formed by the initial position and $p_0$ sampled from the auxiliary distribution. Then the proposal is accepted with probability $A_{K,\eta} = \min \{1, \bbf(q_K)\bbg(p_K) / \bbf(q_0)\bbg(p_0)\}$, that is 
\begin{align}
\label{eqn:acceptance_probability}
\log A_{K,\eta}(q_0, p_0):=& \min\left\{ 0, {\Hamil(q_0, p_0)} - \Hamil(q_K, p_K) \right\}.
\end{align}

\textbf{Stochastic Gradient HMC.}
As discussed in~\cite{ZouGu2021}, while the functions of the auxiliary distribution can be more easily calculated since it is chosen by the users, calculations of gradients of the target density are not always readily available or can be attained with low costs. Therefore, they have been approximated in practice, here are a few examples:
\emph{Mini-batch stochastic gradient:} in this case, $U(\bbq)= \sum_{i=1}^nU_i(\bbq)$, $\xi$ is a random variable for uniformly randomly pick a size-$B$ subset $\calI$ of $[n]$, then the gradient of $U(\bbq)$ is estimated by the following unbiased estimator,
${\tilde G_r}(\bbq,\xi)=\frac{n}{B}\sum_{i\in \calI}\nabla U_i(\bbq)$.
Variations of the above method include \emph{stochastic variance reduced gradient}, \emph{stochastic averaged gradient}, and \emph{control variate gradient}. Details can be found in~\cite{ZouGu2021} and references therein. 

In this paper, we assume that 
\begin{asm}
\label{asm:SG_conv}
The approximate calculation of $\nabla U$ is treated as a distribution, denoted as $\nabla U^\omega$. Furthermore, there exist $0 < \underline{\ell}\le {\bar L}<\infty$ and ${\bar T}>0$, such that, $U^\omega \in \calS_{\ell^\omega, L^\omega}(\Real^d)$ almost surely, with $\ell^\omega \ge \underline{\ell}>0$ and $L^\omega\le {\bar L}<\infty$ and $\|\nabla^3 U^\omega\|\le {\bar T}$ almost surely.
\end{asm}

\textbf{AD-HMC.}
The possibility of utilizing general asymmetric auxiliary distributions $\bbg(\bbp)$ affords us a modification of SGHMC Algorithm by a procedure alternating Hamiltonian motion in forward and backward directions for the same length $T$. The modified is called the Alternating Direction HMC (AD-HMC). One of the motivations for such modification is to produce symmetry with asymmetric auxiliary distribution, which corresponding to self-adjointness of the underlying operator.

\textbf{Reversibility for AD-HMC.}
One step of the proposed AD-HMC algorithm for asymmetrical auxiliaries $\bbg$  
starts from a $q_0\in \Real^d$ by generating a sample $p_0\in \Real^d$ and applying forward Hamiltonian motion that then carries the pair $(q_0, p_0)$ to some $(q_1,p_{01})$. Then, another momentum $p_{12}\in \Real^d$ is sampled and the backward Hamiltonian motion carries $(q_1,p_{12})$ to $(q_2, p_2)$, yielding the candidate $q_2$ for the next state. Similarly, should we start AD-HMC with $q_2$, the pair of momentum vectors $p_2$ and $p_{01}$ will take us back to $q_0$ through $q_1$. 
So, we will accept the proposed move to $q_2$ with probability
\begin{align}
\label{defn:mh_adhmc}
\cP(q_0,q_1,q_2)=\min \left\{1, \frac{\distribution(q_2)\bbg(\Pi_{q_1}^{-b}(q_0))\bbg(\Pi_{q_2}^{-f}(q_1))}{\distribution(q_0)\bbg(\Pi_{q_0}^{-f}(q_1))\bbg(\Pi_{q_1}^{-b}(q_2))}\right\}\,,
\end{align}
where $\Pi^{-f}_{q_0}(q_1)$ denotes the momentum of the forward motion carries $q_0$ to $q_1$, and $\Pi_{q_1}^{-b}(q_2)$ the momentum of the backward motion carries $q_1$ to $q_2$. 
The transition probability of the AD-HMC Markov chain with the Hamiltonian motion augmented with the MH rejection step using~\eqref{defn:mh_adhmc} is equal to
$
\pr(q_0, q_2)=
\int_{\Real^d}\cP(q_0,q_1,q_2)\bbg(\Pi_{q_0}^{-f}(q_1))\bbg(\Pi_{q_1}^{-b}(q_2))\,dq_1
$.
It is established in~\cite{ADHMC} that the underlying Markov chain of the augmented AD-HMC procedure possesses the desired time reversibility. 

\subsection*{Preliminary Results.}

\begin{lem}
\label{lem:moment_of_UV} Under the Assumption~\ref{asm:matrixspace}, we have, for any integer $p > 1$, 
\begin{align}
\label{eqn:moment_of_U}
\left[\ex_{q\sim \exp(-U)} \|\nabla U(q)\|_2^{2p}\right]^{\frac{1}{p}}\le &(d+2p-2) L_U, \\
\label{eqn:moment_of_V} \left[\ex_{p\sim \exp(-V)} \|\nabla V(p)\|_2^{2p}\right]^{\frac{1}{p}}\le & (d+2p-2) L_V.
\end{align}
\end{lem}
\begin{proof}
Assumption~\ref{asm:matrixspace} implies that $Tr(\nabla^2 U) \le L_U d$, \eqref{eqn:moment_of_U} follows from Lemma 9 of~\cite{ChenGatmiry}, note that the subexponential condition is naturally satisfied. Meanwhile,  by an application of the Green's formula and H\"older's inequality, similar to that in the proof of Lemma 9 in~\cite{ChenGatmiry}, \eqref{eqn:moment_of_V} follows.
\end{proof}
\begin{lem}
If, in addition to Assumption~\ref{asm:matrixspace}, the auxiliary distribution satisfies $\ex[p_i]=0$, $\ex[p_ip_j]=\delta_{ij} \sigma_i^2$, and there exists  $\iota>0$ such that $\ex[p^4_i]\le \Sigma_4$ and $\sigma_i^2\le \sigma_2$ for any $i,j=1,2,\ldots, d$. we have, $\ex_{p\sim \exp(-V)}[(p^T\nabla^2U(q)p)^2] \le (\Sigma_2+\Sigma_4) d L^2_U  $.
\end{lem}
\begin{proof}
Direct calculations give us
\begin{align*}
\ex_{p\sim\exp(-V)}[(p^T\nabla^2U(q)p)^2] = &\ex \sum_{i,j=1, i\neq j}^d \left[\frac{\partial^2}{\partial q_i\partial q_j} U(q)\right]^2 \sigma_i^2\sigma_j^2  + \sum_{i=1}^d \left[\frac{\partial^2}{\partial q_i^2}U(q)\right]^2\ex[p_i^4]\\ \le & (\Sigma_2+\Sigma_4) \|\nabla^2 U(q)\|^2_F\le (\Sigma_2+\Sigma_4) d L^2_U .
\end{align*} The last inequality follows from the Assumption~\ref{asm:matrixspace}. More specifically, 
for any two real positive semi-definite matrices $A$ and $B$ satisfying $A\preceq B$, then we know that $\|A\|_F\le \|B\|_F$. (A quick proof $Tr(B^2) \ge 2Tr(AB) -Tr(A^2) = Tr(A^2) +2 (Tr(A(B-A)) \ge  Tr(A^2)$ where the first inequality is due to $Tr(A-B)^2\ge 0$ and the second one since both $A$ and $B-A$ are symmetric and positive semidefinite.)
\end{proof}

\section{Geometric Convergence of SGHMC implementations in Euclidean spaces}
\label{sec:geometric_convergence}

In this section, explicit geometric convergence rates are estimated for general HMC algorithms, including features such as general auxiliary distribution (ADHMC) and stochastic gradient estimation(SGHMC). We are using a \emph{functional approach}.

A basic argument for geometric convergence of Markov chain, treated as iterations driven by the Markov operator, including explicit estimation of convergence rates, based on analyzing functional displacement of the operator is presented in~\cite{https://doi.org/10.1002/rsa.3240040402}. 
The key result is that given a \emph{time-reversible} Markov chain, with conductance $\Phi$, the inequality for the inner product $\langle h, Mh\rangle \le \left(1-\frac{\Phi^2}{2}\right)\|h\|^2$ holds for every mean zero, non constant $h\in \Lt$,
where $Mh$ represents the image of $h$ under the Markov operator, more precisely, $Mf(q) = \int_{\Real^d} f(q) P(q, dq') dq'$. Rewriting the inequality, we have $\|h\|^2 - \langle h, Mh\rangle \ge \frac{\Phi^2}{2}\|h\|^2$,
which says that the norm of the displacement of the Markov chain is lower bounded by norm of the preimage up to a constant. Subsequently, the convergence rate of the Markov chain to its invariant measure can be quantified utilizing the above inequality, see e.g.~\cite{10.1214/ECP.v2-981},~\cite{Roberts_Tweedie_2001} and~\cite{10.1214/aoap/1026915617}. More specifically,
\begin{lem}
\label{lem:convergence_rate}
If $\|h\|^2 - \langle h, Mh\rangle \ge \frac{\Phi^2}{2}\|h\|^2$ is satisfied by any mean zero $h\in \Lt$, the time reversible Markov chain converges to its stationary distribution at a rate no smaller than $1- \frac{\Phi^2}{2}$ in total variational distance.
\end{lem}
\begin{defn}
\label{defn:convergence_rate} For a $r\in(0,1)$, a Markov chain is said to converge exponentially to its stationary distribution with rate no smaller than $r$ if there exists a $C>0$ such that for all $n\ge 1$, we have $d(\pi_n, \pi_\infty)\le Cr^n$ for certain distance between (probability) measures.
\end{defn}


Define the mean joint partial derivatives $\mu_{ij} =  \int_{\Real^d} \frac {\partial V(p)}{\partial p_i} \frac {\partial V(p)}{\partial p_j} \bbg(p) dp $ and the mean second derivative $\sigma_{ij} =  \int_{\Real^d} \frac {\partial^2 V(p)}{\partial p_i \partial p_j}  \bbg(p) dp$. Utilizing Lemma~\ref{lem:convergence_rate}, we have:
\begin{lem}
\label{lem:displacement}
Under Assumption~\ref{asm:matrixspace} and that $\sigma_{ij}=\mu_{ij}$ for $i,j=1,2, \ldots d$.
Suppose that $M_H$ represents the Markov operator generated by the HMC algorithm with $K$ leapfrog steps, then, 
\begin{align}
\label{eqn:displacement}
\|h\|^2-\langle h, M_Hh\rangle \ge  K\eta^2\left(\frac{C_1\sigma_{V}^2}{2} -A_3\eta\right)\|h\|^2.
\end{align}
where $C_1$ is a constant determined by Poincar\'e inequality for general measure, $\sigma_{V}^2:= \int_{\Real^d}\|\nabla V(p) \|_2 g(p)dp$, and the constant $A_3$ is defined in Lemma~\ref{lem:acceptance_prob}. 
\end{lem}
The proof of Lemma~\ref{lem:displacement} can be found in Sec.~\ref{sec:proof_of_key_lemma}. It leads to the following result:
\begin{thm}
\label{thm:main_convergence_rate}
Under Assumption~\ref{asm:matrixspace} and $\sigma_{ij}=\mu_{ij}$ for $i,j=1,2, \ldots d$, for $\eta< \frac { C_1\sigma_{V}^2}{4A_3}$, the Markov chain generated by the HMC algorithm converges at a rate no smaller than $\frac{K\eta^3\sigma_{V}^2}{4}$, more precisely,
\begin{align*}
d_{TV}({\hat \pi}^n,\pi) \le \left(1-\frac{K\eta^3\sigma_{V}^2}{4}\right)^nd_{TV}({\hat \pi},\pi),
\end{align*}
with ${\hat \pi}$ denotes the initial distribution of the Markov chain and ${\hat \pi}^n$ denotes its distribution after $n$ transitions.
\end{thm}
\begin{proof}
The theorem follows from Lemmata~\ref{lem:convergence_rate} and~\ref{lem:displacement}. 
\end{proof}
In case of SGHMC, we have: 
\begin{lem}
\label{lem:displacement_SGHMC}
Under Assumption~\ref{asm:matrixspace} and ~\ref{asm:SG_conv}, with $\sigma_{ij}=\mu_{ij}$ for $i,j=1,2, \ldots d$, suppose that $M_{SG}$ represents the Markov operator generated by the SGHMC algorithm with $K$ leapfrog steps, then, 
\begin{align*}
\|h\|^2-\langle h, M_{SG}h\rangle \ge  K\eta^2\left(\frac{C_1\sigma_{V}^2}{2}  - A_3^{SG}\eta\right)\|h\|^2,
\end{align*}
where $A_3^{SG}$ is the constant from Lemma~\ref{lem:acceptance_prob_SG}.
\end{lem}
\begin{proof}
The only difference from Lemma~\ref{lem:displacement} are the constants estimated by Lemma~\ref{lem:error_general_SG}. 
\end{proof}
\begin{cor}\label{cor:main_convergence_rate_SG}
Under Assumptions~\ref{asm:matrixspace} and~\ref{asm:SG_conv} and $\sigma_{ij}=\mu_{ij}$ for $i,j=1,2, \ldots d$, for $\eta < \frac {C_1\sigma_{V}^2}{4A^{SG}_3}$, the Markov chain generated by the SGHMC algorithm converge at a rate no smaller than $\frac{K\eta^3\sigma_{V}^2}{4}$, more precisely,
\begin{align*}d_{TV}({\hat \pi}^n,\pi) \le \left(1-\frac{K\eta^3\sigma_{V}^2}{4}\right)^nd_{TV}({\hat \pi},\pi).
\end{align*}
\end{cor}

\begin{rem}
As we can see from Lemma~\ref{lem:displacement}, Theorem~\ref{thm:main_convergence_rate} and their proofs, the key for determining the convergence rates of the SGHMC algorithms lies in quantifying the closeness of the \emph{symplectic integrator} to exact solution to the Hamiltonian system. These \emph{quantifications} have been summarized in the next section, and detailed calculations are presented in Sec. \ref{sec:approximation_error}. 

Both Theorem~\ref{thm:main_convergence_rate} and Corollary~\ref{cor:main_convergence_rate_SG} apply to AD-HMC algorithms directly.

For background and derivations, as well as estimates of constant for the \emph{Poincar\'e inequality} for a general family of measures that include log-concave case, see, e.g.~\cite{10.1214/ECP.v13-1352} and~\cite{villani2009hypocoercivity}. Due to this connection, as well as conditions on the moments of the auxiliary distributions, our convergence rates are less restricted by large dimension, comparing to for example those in~\cite{10.5555/3327345.3327502} and~\cite{ChenGatmiry}. 

It is desirable to obtain more precise results on the rate, this will be depend on more sophisticated on advances in quantitative results on functional inequalities including the Poincar\'e's inequality.
\end{rem}

\section{Ramifications}
\label{sec:ramifications}

In this section, we present quantitative results on some key aspects of general HMC algorithms. First, we provide a range of \emph{quantitative estimations} of potential numerical errors in leapfrog implementations of the algorithms. Second, we will characterize the statistical distance in \emph{KL divergence} between two proposed HMC steps with respect to the distance between their initial states. Lastly, bounds on the (Metropolis-Hastings) \emph{acceptance probability} are obtained. From the dependence of our main results in last section on some of these results, we can see that that these quantities are crucial for the performance of SGHMC. In addition, The results presented here can also be utilized for establishing quantitative convergence results through other arguments, for example through \emph{conductance} type arguments as presented  in~\cite{ChenGatmiry}.

\textbf{Leapfrog vs Exact.}
One of the keys to the success of HMC algorithms is the effectiveness of the numerical symplectic integration of the Hamiltonian differential equations. Extensive efforts have been devoted to such studies for HMC Gaussian auxiliary distribution across the existing literature. Allowing \emph{general auxiliary distributions} certainly has made the analysis more complicated; the introduction of stochastic gradient estimation leads to new difficulties in this problem. In a series of technical results, we are able to provide sharp estimates on the error produced in these numerical procedures. 
\begin{lem}
\label{lem:error_general}
Under Assumptions~\ref{asm:matrixspace} and~\ref{asm:third_derivative},
when $\sup_{q \in \Real^d}\Big\|\nabla^3 U(q)\Big\|<\infty$ and $\sup_{p \in \bbP}\Big\|\nabla^3 V(p)\Big\|<\infty$,  we have, 
\begin{align}
|||Q(\eta)- {\hat q}|||_2\le&\left\{\frac{\sup_{p \in \bbP}\Big\|\nabla^3V(p)\Big\|_{op} (d+2)L_U}{24}+\frac{L_VL_U [(d+2)L_V]^{\frac12}}{6}\right\}\eta^3\,,
\nonumber
\\
|||P(\eta)- {\hat p}|||_2\le&\left(L_V(d+2)\right)^{1/2}\times
\nonumber 
\\
&\times \left\{\frac{\sup_{q \in \Real^d}\Big\|\nabla^3U(q)\Big\|(L_U)^{3/2} (d+2)^{1/2}+(L_U)^{3/2}(L_V)^{1/2}}{6} + \right.
\nonumber  
\\ 
&\left. +  \frac{\sup_{q \in \Real^d}\Big\|\nabla^3U(q)\Big\|}{12}+\frac{(L_U)^{3/2}(L_V)^{1/2}}{4}\right\}\eta^3.\nonumber
\end{align}
\end{lem}
In case of stochastic gradient HMC,  Assumptions~\ref{asm:SG_conv} allows the exchange of limit and expectation, hence,
\begin{lem}
\label{lem:error_general_SG}
Under Assumptions~\ref{asm:SG_conv}, when $\sup_{q \in \Real^d}\Big\|\nabla^3 U(q)\Big\|<\infty$ and $\sup_{p \in \bbP}\Big\|\nabla^3 V(p)\Big\|<\infty$, we have, 
\begin{align}
|||Q(\eta)- {\hat q}|||_2\le&\left\{\frac{\sup_{p \in \bbP}\Big\|\nabla^3V(p)\Big\|_{op} (d+2)\ex[L_U^\omega]}{24}+\frac{L_V\ex[L_U^\omega] [(d+2)L_V]^{\frac12}}{6}\right\}\eta^3,
\nonumber
\\
|||P(\eta)- {\hat p}|||_2\le&(L_V(d+2))^{1/2}\times\nonumber 
\\ 
&\times\left\{\frac{\sup_{q \in \Real^d}\Big\|\nabla^3U(q)\Big\|\ex[(L_U^\omega)^{3/2}] (d+2)^{1/2}+\ex[(L_U^\omega)^{3/2}](L_V)^{1/2}}{6} + \right.
\nonumber  
\\ 
&\left. +  \frac{\sup_{q \in \Real^d}\Big\|\nabla^3U(q)\Big\|}{12}+\frac{\ex[(L_U^\omega)^{3/2}](L_V)^{1/2}}{4}\right\}\eta^3\,.
\nonumber
\end{align}
\end{lem}

\textbf{Continuity of the Step in probability space.}
A key observation for HMC is that for a pair of starting points $q_1$ and $q_2$, the distance of probability measures (for example the KL divergence) for the next step of the algorithm is bounded by the linear order of the distance between $q_1$ and $q_2$. Therefore when these two points are close, the next step is also similar. Hence, it is easy to see that this is a key step in a conductance based argument for geometric convergence, as seen in~\cite{ChenGatmiry}.

\textbf{Density of Pushforward Auxiliary Distributions.}
For a fixed $\bbq \in \Real^d$, the probability measure $\calP_q$ of the image $Q\in \Real^d$ can be viewed as a pushforward of the auxiliary probability measure via the integrator. Its density  is given by \eqref{eqn:denisity of Pq},
where $\bbp(\bbq, Q)$ denotes the inverse of the integrator (as $\Pi^{-f}_{Q} (\bbq)$ in~\eqref{defn:mh_adhmc}),
\begin{align}\label{eqn:denisity of Pq}
\bbg (\bbp(\bbq, Q)) \det \left( \frac{\partial \bbp(\bbq, Q)}{\partial Q}\right)\,.
\end{align}

\textbf{Kullback-Leibler(KL) divergence calculations.}
For any pair $\bbq_1, \bbq_2\in \Real^d$, the Kullback-Leibler(KL) divergence $KL (\calP_{\bbq_1}|| \calP_{\bbq_2})$ can be written as, 
\begin{align*}
&KL (\calP_{\bbq_1}|| \calP_{\bbq_2})
=
\int_{\Real^d} \left\{ \log\bbg (\bbp)-\log[\bbg (\bbp(\bbq_2, Q(\bbq_1, \bbp)))]-\log 
\det \left( \frac{\partial \bbp(\bbq_2, Q(\bbq_1, \bbp)}{\partial Q}\right)
\right\} \bbg (\bbp) d\bbp,
\end{align*}

\begin{lem}
\label{lem:continuity_general}
For HMC with general auxiliary distribution, we have,
\begin{align*}
KL (\calP_{\bbq_1}|| \calP_{\bbq_2}) \le \frac{\eta\|\nabla^3V\|L_U}{2}\|q_1- q_2\|.
\end{align*}
\end{lem}
In addition, we have the similar result for stochastic gradient estimate,
\begin{lem}
For SGHMC, we have
\label{lem:continuity_SG}
\begin{align*}
KL (\calP_{\bbq_1}|| \calP_{\bbq_2}) \le \frac{\eta\|\nabla^3V\|\ex[L_U^\omega]}{2}\|q_1- q_2\|.
\end{align*}
\end{lem}
The proof of Lemmata~\ref{lem:continuity_general} and~\ref{lem:continuity_SG} can be found in Sec.\ref{sec:continuity}.

\textbf{Lower bounding the Acceptance Probability.}
Another key component in all convergence analysis of HMC with numerical integrator is the estimation (bounding) of the acceptance probability of the proposed motion. This is eventually reduced to the estimation of the deviation of the numerical integrator from the exact solution in terms of the function value. More specifically, we have,
\begin{lem}
\label{lem:acceptance_prob}
Under Assumptions~\ref{asm:matrixspace} and~\ref{asm:third_derivative}, we have
$\ex|U({\hat q})-U(Q(\eta))|+\ex|V({\hat p})-V(P(\eta))|\le A_3\eta^3$, with
\begin{align*}
A_3:=&[L_U\sigma_q+(dL_U)^{1/2}]\left\{\frac{(d+2)T_V L_U}{24}+\frac{L_VL_U [(d+2)L_V]^{\frac12}}{6}\right\}
\\& +[L_V\sigma_p+(dL_U)^{1/2}](L_V(d+2))^{1/2}\left\{\frac{T_U(L_U)^{3/2} (d+2)^{1/2}+(L_U)^{3/2}(L_V)^{1/2}}{6}\right.
\\& \left.+ \frac{T_U}{12}+\frac{(L_U)^{3/2}(L_V)^{1/2}}{4}\right\}.
\end{align*}
\end{lem}
Similarly, for SGHMC, we have,
\begin{lem}
\label{lem:acceptance_prob_SG}
Under Assumptions~\ref{asm:matrixspace} and~\ref{asm:third_derivative} for $V$ and Assumptions~\ref{asm:SG_conv} for the stochastic implementations, we have
$\ex|U({\hat q})-U(Q(\eta))|+\ex|V({\hat p})-V(P(\eta))|\le A_3^{SG}\eta^3$, with
\begin{align*}
A_3^{SG}:=&\left\{\frac{(d+2)T_V [\ex[(L_U^\omega)^2]+d^{1/2}\ex[L_U^\omega)^{3/2}]}{24}+\frac{L_V[\ex[(L_U^\omega)^2]+d^{1/2}\ex[(L_U^\omega)^{3/2}] [(d+2)L_V]^{\frac12}}{6}\right\}
\\& +L_V\sigma_p(L_V(d+2))^{1/2}\left\{\frac{{\bar T}\ex[L_U^\omega)^{3/2}]] (d+2)^{1/2}+\ex[L_U^\omega)^{3/2}](L_V)^{1/2}}{6}\right. \\ &\left. \quad+ \frac{{\bar T}}{12}+\frac{\ex[L_U^\omega)^{3/2}](L_V)^{1/2}}{4}\right\}
\\&+d^{1/2}\left\{\frac{{\bar T}\ex[(L_U^\omega)^2] (d+2)^{1/2}+\ex[(L_U^\omega)^2](L_V)^{1/2}}{6}+ \frac{{\bar T}\ex[L_U^\omega)^{1/2}]}{12}+\frac{\ex[(L_U^\omega)^2](L_V)^{1/2}}{4}\right\}.
\end{align*}
\end{lem}
The proofs of these Lemmata can be found in Sec.\ref{sec:acceptance_probability}. This naturally leads to the following result. 
\begin{pro}
\label{thm:acceptance_prob}
Under Assumptions~\ref{asm:matrixspace} and~\ref{asm:third_derivative} ( Assumptions~\ref{asm:matrixspace} and~\ref{asm:third_derivative} for $V$ and Assumptions~\ref{asm:SG_conv} for SGHMC):\\ 
For any $\varrho, \delta\in (0,1)$, one can choose $K$ and $\eta$ such that for subset $D\subseteq \Real^d\times \Real^d$ and $\pr[(q,p)\in D]\ge 1-\delta$, the acceptance probability is lower bounded by $\varrho$.
\end{pro}


\begin{rem}
Results from Lemmata~\ref{lem:error_general}, \ref{lem:error_general_SG}, \ref{lem:continuity_general}, \ref{lem:continuity_SG} and~\ref{thm:acceptance_prob} are key elements of various other approaches, 
such as the \emph{conductance analysis} in~\cite{ChenGatmiry} 
or the \emph{probabilistic analysis} in~\cite{10.5555/3327345.3327502}.\\ 
Coupled with proper extra conditions, such as those on conductance, similar results to those in~\cite{ChenGatmiry} can be obtained readily utilizing the estimations we presented here. 
\end{rem}

\section{Conclusions}
\label{sec:conclusions}

In this paper, we analyzed the convergence of stochastic gradient HMC with alternating direction and leapfrog implementations via an analytic method. As more applications of HMC emerge from different areas of machine learning, we expect these results to allow the presented algorithms to be adapted more readily and with higher confidence. We also expect the analytic methods developed in the paper can be more extensively utilized in the analysis of algorithms in this domain. 

\bibliographystyle{abbrv}

\appendix

\vskip 1cm

In the appendix, we will present first the detailed description of both SGHMC and SGHMC with Alternating direction in Sec.~\ref{sec:algorithms}. Then, we present some fundamental estimations that are both crucial for establishing the main results in the paper and of independent interests. First, in Sec.~\ref{sec:approximation_error}, we present a detailed estimation between the exact Hamiltonian calculation and the leapfrog integrator; then, in Sec.~\ref{sec:acceptance_probability}, we provide a lower bound to the acceptance probability. Then, the proof of the key Lemma~\ref{lem:displacement} is presented in \ref{sec:proof_of_key_lemma}. In Sec.~\ref{sec:continuity} we quantify the dependence of the leapfrog implementations upon the initial state in terms of divergence of probability measure.

\section{HMC and AD-HMC Algorithms}
The two main HMC algorithms under consideration in this article are presented here. Algorithm~\ref{algo:hmc} presents the standard HMC with Gaussian auxiliaries (and kinetic term $V(p) = p^2/2$) and utilizes a stochastic oracle to obtain noisy estimates $\nabla^{\omega} U(\cdot)$ for the gradient of the potential $U(q)$. It also uses $K$ steps of the symplectic integrator each with size $\eta$ to implement Hamiltonian motion. 
\label{sec:algorithms}
\vskip -0.2in
\begin{algorithm}[H]
\begin{algorithmic}
 \STATE \textbf{Initialization:} stochastic oracle for $\nabla U^{\omega}(q) $ for potential energy $U(q)$ gradient; kinetic energy $V(p) = p^2/2$ with gradient $\nabla V(p) = p$; initial iterate $q_{0}$; $K$ steps of size $\eta$ for total trajectory length $K\eta$ 
 \FOR{$n=1,\ldots$}
  \STATE Set $ q_0 \leftarrow q_{n-1}$\;
  \STATE Sample: $p_{0}\sim \bbg(p)$  \;  
  \STATE Lift: $(q_{0}, p_0) \leftarrow q_{0}$ \;
  \STATE Move:     
  \STATE Start with sample of $\nabla U^{\omega}(q_0)$ 
  \FOR{$k=0,\ldots, K-1$}
    \STATE Set $p_{k+\frac 1 2} \leftarrow p_{k} -\frac{\eta} 2 \nabla U^{\omega}(q_k)$ 
    \STATE Set $ q_{k+1} \leftarrow q_k +\eta \nabla V(p_{k+\frac 12})$
    \STATE Sample $\nabla U^{\omega}(q_{k+1})$ 
    \STATE Set ${p_{k+1}} \leftarrow p_{k+\frac 12} - \frac {\eta} 2 \nabla U^\omega(q_{k+1}). $
  \ENDFOR
\STATE Sample $Z\sim $ Uniform $(0,1)$. 
\IF{ $Z \le 
\frac{\distribution(q_K)\bbg(p_K)}{\distribution(q_0)\bbg(p_0)}\,
$}
  \STATE Project: $q_{n} \leftarrow (q_K,p_K)$ \;
  \ELSE
  \STATE Set $q_{n} \leftarrow q_0$ 
  \ENDIF
 \ENDFOR
\end{algorithmic}
\caption{SGHMC}
\label{algo:hmc}
\end{algorithm}
\hfill
\vskip -0.2in
\begin{algorithm}[H]
\begin{algorithmic}
\STATE \textbf{Initialization:} stochastic oracle for $\nabla U^{\omega}(q) $ for potential energy $U(q)$ gradient; kinetic energy $V(p)$ with gradient oracle $\nabla V(p)$; initial iterate $q_{0}$; $K$ steps of size $\eta$ for total trajectory length $K\eta$ 
\FOR{$n=1,\ldots, N$}
\STATE  Set $ q_0 \leftarrow q_{n-1}$\;
\STATE $\qquad\qquad\qquad\qquad $ \COMMENT{{\it forward motion}}
\STATE  Sample: $p_{0}\sim \bbg(p) \qquad$    
\STATE  Lift: $(q_{0}, p_0) \leftarrow q_{0}$ \;
\STATE Move:     
\STATE Start with sample of $\nabla U^{\omega}(q_0)$ 
  \FOR{$k=0,\ldots, K-1$}
    \STATE Set $p_{k+\frac 1 2} \leftarrow p_{k} -\frac{\eta} 2 \nabla U^{\omega}(q_k)$ 
    \STATE Set $ q_{k+1} \leftarrow q_k +\eta \nabla V(p_{k+\frac 12})$
    \STATE Sample $\nabla U^{\omega}(q_{k+1})$ 
    \STATE Set ${p_{k+1}} \leftarrow p_{k+\frac 12} - \frac {\eta} 2 \nabla U^\omega(q_{k+1}). $
  \ENDFOR
\STATE  Project: $q'_{0} \leftarrow (q_K, p_K)$ \;
\STATE $\qquad\qquad\qquad\qquad $ \COMMENT{{\it backward motion}}
\STATE  Sample: $p'_{0}\sim \bbg(p)\quad$ 
\STATE  Lift: $(q'_{0},p'_{0}) \leftarrow q'_{0}$ \;
  \STATE Move:     
  \STATE Start with sample of $\nabla U^{\omega}(q'_{0})$ 
  \FOR{$k=0,\ldots, K-1$}
    \STATE Set $p'_{k+\frac 1 2} \leftarrow p'_{k} \;\mathbf{+} \;\frac{\eta} 2 \nabla U^{\omega}(q'_{k})$ 
    \STATE Set $ q'_{k+1} \leftarrow q'_k \;\mathbf{-}\;\eta \nabla V(p'_{k+\frac 12})$
    \STATE Sample $\nabla U^{\omega}(q'_{k+1})$ 
    \STATE Set ${p'_{k+1}} \leftarrow p'_{k+\frac 12} \mathbf{+} \frac {\eta} 2 \nabla U^\omega(q'_{k+1}). $
  \ENDFOR
\STATE Sample $Z\sim $ Uniform $(0,1)$. 
\IF{ $Z \le 
\frac{\distribution(q'_K)\bbg(p_K)\bbg(p'_K)}{\distribution(q_0)\bbg(p_0)\bbg(p'_{0})}\,
$}
    \STATE  Project: $q_{n} \leftarrow (q'_K, p'_K)$ \;
  \ELSE
  \STATE Set $q_{n} \leftarrow q_0$ 
  \ENDIF
\ENDFOR
\end{algorithmic}
\caption{Stochastic Gradient AD-HMC}\label{algo:adhmc}
\end{algorithm}
%
Algorithm~\ref{algo:adhmc} presents the Alternating Direction HMC algorithm that can be used with arbitrary auxiliaries (with kinetic term $V(p)$). A stochastic oracle $\nabla U^\omega(q)$ provides an estimate of the gradient of the potential $U(q)$. Two sets of Hamiltonian motions are implemented using $K$ steps of the symplectic integrator of size $\eta$ : the first implements forward motion, and the second implements backward motion. Note that a new sample $p'_0$ of the momentum is used for the latter set of motion. The MH correction steps involve both sets of initial momenta $p_0$ and $p'_0$ as well as the last momenta $p_K$ and $p'_K$.

\section{Error Estimation for Leapfrog Gradient Calculations}
\label{sec:approximation_error}

In this section, we quantify the errors between the leapfrog implementation, including the case with stochastic gradient, and the exact Hamiltonian solutions. These results will be used later in multiple occasions, such as the estimation of the lower bound on acceptance probabilities and eventual convergence rates.

\textbf{Error Estimation for Exact Gradient Calculation with General Auxiliary in One Leapfrog Step}

In this section, we will carry out a similar error estimation for leapfrog symplectic integrator for general auxiliary distributions, but just an one-step leapfrog symplectic integrator for the Hamiltonian equation. Recall that the update takes the following form,
\begin{align}
{\hat q} =& q+\eta \nabla V\left(p -\frac12 \eta \nabla U(q)\right), \label{eqn:LF_q}\\
{\hat p} =& p-\frac12 \eta \nabla U(q) - \frac12 \eta \nabla U({\hat q}).\label{eqn:LF_p}
\end{align}

Meanwhile, the exact trajectory follows,
\begin{align*}
{\dot Q} =\nabla V(P),\quad {\dot P} =-\nabla U(Q); \qquad Q(0) = q,\quad P(0) =p.
\end{align*}
Hence, we have,
\begin{align}
\label{eqn:exact_q}
Q(\eta) = q+ \int_0^\eta \nabla V(P(t)) dt
\\
\label{eqn:exact_p}
P(\eta) = p- \int_0^\eta \nabla U(Q(t)) dt.
\end{align}

\textbf{Examination of the difference between ${\hat q}$ and $Q(\eta)$}
\label{sec:q_diff_new}
We have:
\begin{align*}
{\hat q}-Q(\eta) =& \int_0^\eta \left[\nabla V\left(p -\frac12 \eta \nabla U(q)\right)- \nabla V(P(t))\right] dt  \\ &\hbox{\quad follows from equations~\eqref{eqn:LF_q} and ~\eqref{eqn:exact_q}}\\
=&\int_0^\eta \int_0^1 \nabla^2 V \left(s\left(p -\frac12 \eta \nabla U(q)\right)+(1-s)P(t)\right)\left[\left(p -\frac12 \eta \nabla U(q)\right)-P(t)\right] \;ds \;dt \\&\hbox{\quad an application of Newton-Leibnitz}\\ =&\int_0^\eta \int_0^1 \nabla^2 V \left(s\left(p -\frac12 \eta \nabla U(q)\right)+(1-s)P(t)\right)\left(\int_0^t\nabla U(Q(\tau))d\tau-\frac12 \eta \nabla U(q)\right) \;ds \;dt,
\\ &\hbox{\quad follows from equations~\eqref{eqn:LF_p} and ~\eqref{eqn:exact_p}}\\ =&\underbrace{\int_0^\eta \int_0^1 \nabla^2 V \left(s\left(p -\frac12 \eta \nabla U(q)\right)+(1-s)P(t)\right)\left(\int_0^t\nabla [U(Q(\tau))-\nabla U(q)] d\tau\right) \;ds \;dt}_{A_1}+\\ 
&+\underbrace{\int_0^\eta \int_0^1 \nabla^2 V \left(s\left(p -\frac12 \eta \nabla U(q)\right)+(1-s)P(t)\right)\left(t-\frac12 \eta \right)\nabla U(q) \;ds \;dt}_{A_2}
\\ & \hbox{\quad write the term $\frac12 \eta \nabla U(q)$ as $\frac12 \eta \nabla U(q)-\int_0^t \nabla U(q) d\tau + t\nabla U(q)$.}
\end{align*}

Now, examine $A_2$
\begin{align*}
A_2=&\int_0^\eta \int_0^1 \nabla^2 V \left(s\left(p -\frac12 \eta \nabla U(q)\right)+(1-s)P(t)\right)\left(t-\frac12 \eta \right)\nabla U(q) \;ds \;dt, 
\\=&
\int_0^1\int_0^\eta  \nabla^2 V \left(s\left(p -\frac12 \eta \nabla U(q)\right)+(1-s)P(t)\right)\left(t-\frac12 \eta \right)\nabla U(q) \;dt\;ds 
\\ &
\hbox{\quad exchange the order of integrations} 
\\=&
\int_0^1\int_0^\eta  \left[\nabla^2 V \left(s\left(p -\frac12 \eta \nabla U(q)\right)+(1-s)P(t)\right)-\nabla^2 V \left(s\left(p -\frac12 \eta \nabla U(q)\right)+(1-s)p\right)\right]\cdot
\\ &\cdot
\left(t-\frac12 \eta \right)\nabla U(q) \;dt\;ds+
\\ &+
\int_0^1\int_0^\eta  \nabla^2 V \left(s\left(p -\frac12 \eta \nabla U(q)\right)+(1-s)p\right)\left(t-\frac12 \eta \right)\nabla U(q) \;dt\;ds, 
\\=&
\int_0^1\int_0^\eta  \left[\nabla^2 V \left(s\left(p -\frac12 \eta \nabla U(q)\right)+(1-s)P(t)\right)-\nabla^2 V \left(s\left(p -\frac12 \eta \nabla U(q)\right)+(1-s)p\right)\right]\cdot
\\ &\cdot
\left(t-\frac12 \eta \right)\nabla U(q) \;dt\;ds
\\ &
\hbox{\quad because $\nabla^2 V \left(s\left(p -\frac12 \eta \nabla U(q)\right)+(1-s)p\right)$ is independent of $t$ and $ \int_0^\eta (t-\frac12 \eta )dt=0$}
\\=&
\int_0^1\int_0^\eta  \int_0^t\nabla^3 V \left(s\left(p -\frac12 \eta \nabla U(q)\right)+(1-s)P(\tau)\right)\nabla U(Q(\tau))\left(t-\frac12 \eta \right)\nabla U(q) \;d\tau\;dt\;ds
\\&
\hbox{\quad an application of Newton-Leibnitz}
\\=&
\int_0^1  \int_0^\eta \int_\tau^\eta \nabla^3 V \left(s\left(p -\frac12 \eta \nabla U(q)\right)+(1-s)P(\tau)\right)\nabla U(Q(\tau))\left(t-\frac12 \eta \right)\nabla U(q) \;dt\;d\tau\;ds 
\\&
\hbox{\quad exchange the order of integrations with respect to $t$ and $\tau$ }
\\=&
\int_0^1  \int_0^\eta \nabla^3 V \left(s\left(p -\frac12 \eta \nabla U(q)\right)+(1-s)P(\tau)\right)\nabla U(Q(\tau))\frac{\tau}{2}(\eta-\tau) \nabla U(q) \;d\tau\;ds 
\\&
\hbox{\quad integrate out $t$}
\end{align*}
The above calculations can be summarized as 
\begin{lem}
\label{lem:Q_exp_gen} The difference between the leapfrog update with step $\eta$ and the trajectory at time $\eta$ is equal to:
\begin{align}
&{\hat q}-Q(\eta)\nonumber
\\=&\int_0^\eta \int_0^1 \nabla^2 V \left(s\left(p -\frac12 \eta \nabla U(q)\right)+(1-s)P(t)\right)\left(\int_0^t\nabla [U(Q(\tau))-\nabla U(q)] d\tau\right) \;ds \;dt \nonumber \\
&+\int_0^1  \int_0^\eta \nabla^3 V \left(s\left(p -\frac12 \eta \nabla U(q)\right)+(1-s)P(\tau)\right)\nabla U(Q(\tau))\frac{\tau}{2}(\eta-\tau) \nabla U(q) \;d\tau\;ds \label{eqn:Q_exp_gen}
\end{align}
\end{lem}

This will for the basic for quantifying $|||Q(\eta)- {\hat q}|||_2$. 
\begin{lem}
\label{lem:Q_error_general}
Under Assumptions~\ref{asm:matrixspace} and~\ref{asm:third_derivative}, 
we have, 
\begin{align}
\label{eqn:Q_error_general}|||Q(\eta)- {\hat q}|||_2\le\left\{\frac{(d+2)L_UT_V }{24}+\frac{L_VL_U [(d+2)L_V]^{\frac12}}{6}\right\}\eta^3,
\end{align}
\end{lem}
\begin{proof}
Let us look at the two terms in~\eqref{eqn:Q_exp_gen}.
\begin{align*}
&\bigg|\bigg|\bigg|\int_0^\eta \int_0^1 \nabla^2 V \left(s\left(p -\frac12 \eta \nabla U(q)\right)+(1-s)P(t)\right)\left(\int_0^t\nabla [U(Q(\tau))-\nabla U(q)] d\tau\right) \;ds \;dt \bigg|\bigg|\bigg|_2 \\ =& \bigg|\bigg|\bigg|\int_0^\eta \int_0^1 \nabla^2 V \left(s\left(p -\frac12 \eta \nabla U(q)\right)+(1-s)P(t)\right)\left(\int_0^t \int_0^\tau \nabla^2U(Q(w))\nabla V (P(w)) dw; d\tau\right) \;ds \;dt\bigg|\bigg|\bigg|_2\\&\hbox{\quad an application of Newton-Leibnitz}\\ \le& \int_0^\eta \int_0^1\int_0^t \int_0^\tau \bigg|\bigg|\bigg|\nabla^2 V \left(s\left(p -\frac12 \eta \nabla U(q)\right)+(1-s)P(t)\right) \nabla^2V(P(w))\nabla V (P(w))\bigg|\bigg|\bigg|_2dw; d\tau \;ds \;dt
\\ \le& \int_0^\eta \int_0^1\int_0^t \int_0^\tau L_U L_V\bigg|\bigg|\bigg|\nabla V (P(w))\bigg|\bigg|\bigg|_2dw\; d\tau \;ds \;dt
\\&\hbox{\quad due to assumption~\ref{asm:matrixspace}}
\\ \le& L_VL_U [(d+2)L_V]^{\frac12}/6 
\\ & \hbox{\quad due to Lemma\ref{lem:moment_of_UV} and the fact that the joint probability is invariant under the Hamilton motion}
\end{align*}
\begin{align*}
&\bigg|\bigg|\bigg|\int_0^1  \int_0^\eta \nabla^3 V \left(s\left(p -\frac12 \eta \nabla U(q)\right)+(1-s)P(\tau)\right)\nabla U(Q(\tau))\frac{\tau}{2}(\eta-\tau) \nabla U(q) \;d\tau\;ds \bigg|\bigg|\bigg|_2
\\ \le &\int_0^1 \int_0^\eta \|\nabla^3 V\| \frac{\tau}{2}(\eta-\tau) \ex\left[\bigg|\bigg|\nabla U(Q(\tau))\bigg|\bigg|_2\cdot\bigg|\bigg|\nabla U(q) \bigg|\bigg|_2\right] \;d\tau\;ds 
\\ & \hbox{\quad assumption on the operator norm of the third degree tensor}
\\ \le &\int_0^1 \int_0^\eta \|\nabla^3 V\| \frac{\tau}{2}(\eta-\tau) \bigg|\bigg|\bigg|\nabla U(Q(\tau))\bigg|\bigg|\bigg|_2\cdot\bigg|\bigg|\bigg|\nabla U(q) \bigg|\bigg|\bigg|_2 \;d\tau\;ds 
\\& \hbox{\quad H\"older's inequality}
\\ \le & \|\nabla^3 V\| L_U (d+2)/24.
\end{align*}
\end{proof}

\textbf{Examine the difference between ${\hat p}$ and $P(\eta)$}

Again, let us start with a decomposition of the term $P(\eta)- {\hat p}$, 
\begin{align*}
{\hat p}-P(\eta) = &\int_0^\eta \nabla U(Q(t))-\frac12  \nabla U(q) - \frac12  \nabla U({\hat q}) \;dt \\ &\hbox{\quad follows from equations~\eqref{eqn:LF_p} and ~\eqref{eqn:exact_p}}\\=&\int_0^\eta [\nabla U(Q(t)) -\nabla U(q)]\;dt-\frac12 \int_0^\eta [\nabla U(q) -  \nabla U({\hat q})] \;dt\\ =&\int_0^\eta \int_0^t \nabla^2 U(Q(s)) \nabla V(P(s)) \;ds \;dt-\frac12\int_0^\eta  [\nabla U(q) -  \nabla U\left(q+\eta \nabla V\left(p -\frac12 \eta \nabla U(q)\right)\right)] \;dt\\&\hbox{\quad an application of Newton-Leibnitz}\\ = & \underbrace{\int_0^\eta \int_0^t \nabla^2 U(Q(s)) \nabla V(P(s)) \;ds \;dt-\frac12\int_0^\eta  [\nabla U(q) -  \nabla U(q+\eta \nabla V(p))] \;dt}_{B_1} \\ &+\underbrace{\frac12\int_0^\eta  \left[\nabla U(q+\eta \nabla V(p)) -  \nabla U\left(q+\eta \nabla V\left(p -\frac12 \eta \nabla U(q)\right)\right)\right] \;dt}_{B_2}.
\end{align*}

\begin{align*}
B_2=&\frac12\int_0^\eta  \left[\nabla U(q+\eta \nabla V(p)) -  \nabla U\left(q+\eta \nabla V\left(p -\frac12 \eta \nabla U(q)\right)\right)\right] \;dt\\=& \frac12\int_0^\eta \int_0^\eta  \left[\nabla^2 U\left(q+s\nabla V(p)+(1-s)    \nabla V\left(p -\frac12 \eta \nabla U(q)\right)\right) \left[\nabla V(p)-  \nabla V\left(p -\frac12 \eta \nabla U(q)\right)\right]\right] \;ds\;dt
\\&\hbox{\quad an application of Newton-Leibnitz}\end{align*}
$B_2$ is clearly a $\eta^3$ term. 
\begin{align*}
B_1=&\int_0^\eta \int_0^t \nabla^2 U(Q(s)) \nabla V(P(s)) \;ds \;dt-\frac12\int_0^\eta  [\nabla U(q) -  \nabla U(q+\eta \nabla V(p))] \;dt \\
=&\int_0^\eta \int_0^t \nabla^2 U(Q(s)) \nabla V(P(s)) \;ds \;dt-\frac12\int_0^\eta \int_0^\eta  [\nabla^2 U(q+s \nabla V(p))] \nabla V(p) \;dt\\&\hbox{\quad an application of Newton-Leibnitz}\\
=&\underbrace{\int_0^\eta \int_0^t \nabla^2 U(Q(s)) \nabla V(P(s)) \;ds \;dt-\int_0^\eta \int_0^t [\nabla^2 U(q+s \nabla V(p))] \nabla V(P(s)) \;ds \;dt}_{B_{11}}
\\ &+\underbrace{\int_0^\eta \int_0^t [\nabla^2 U(q+s \nabla V(p))] \nabla V(P(s)) \;ds \;dt-
\frac12\int_0^\eta \int_0^\eta  [\nabla^2 U(q+s \nabla V(p))] \nabla V(p) \;dt}_{B_{12}}
\end{align*}
For $B_{11}$, we have, 
\begin{align*}
B_{11}= & \int_0^\eta \int_0^t \nabla^2 U(Q(s)) \nabla V(P(s)) \;ds \;dt-\int_0^\eta \int_0^t [\nabla^2 U(q+s \nabla V(p))] \nabla V(P(s)) \;ds \;dt \\ =& \int_0^\eta \int_0^t [\nabla^2 U(Q(s))-\nabla^2 U(q+s \nabla V(p))] \nabla V(P(s)) \;ds \;dt
\end{align*}
\begin{align*}
B_{12}=&\int_0^\eta \int_0^t [\nabla^2 U(q+s \nabla V(p))] \nabla V(P(s)) \;ds \;dt-
\frac12\int_0^\eta \int_0^\eta  [\nabla^2 U(q+s \nabla V(p))] \nabla V(p) \;dt\\=&
\underbrace{\int_0^\eta \int_0^t [\nabla^2 U(q+s \nabla V(p))] \nabla V(P(s)) \;ds \;dt-\int_0^\eta \int_0^t [\nabla^2 U(q+s \nabla V(p))] \nabla V(p) \;ds \;dt
}_{B_{121}}\\&+ \underbrace{\int_0^\eta \int_0^t [\nabla^2 U(q+s \nabla V(p))] \nabla V(p) \;ds \;dt-
\frac12\int_0^\eta \int_0^\eta  [\nabla^2 U(q+s \nabla V(p))] \nabla V(p) \;dt}_{B_{122}}
\end{align*}
Hence, we have the following expression,
\begin{lem}
\label{lem:P_expression}
\begin{align*}
&P(\eta)- {\hat p}\\= &\int_0^\eta \int_0^t [\nabla^2 \nabla^2 U(Q(s))-U(q+s \nabla V(p))] \nabla V(P(s)) \;ds \;dt\\&+
\int_0^\eta \int_0^t [\nabla^2 U(q+s \nabla V(p))] \nabla V(P(s)) \;ds \;dt-\int_0^\eta \int_0^t [\nabla^2 U(q+s \nabla V(p))] \nabla V(p) \;ds \;dt
\\&+ \int_0^\eta \int_0^t [\nabla^2 U(q+s \nabla V(p))] \nabla V(p) \;ds \;dt-
\frac12\int_0^\eta \int_0^\eta  [\nabla^2 U(q+s \nabla V(p))] \nabla V(p) \;dt\\&+ \frac12\int_0^\eta \int_0^\eta  \left[\nabla^2 U\left(q+s\nabla V(p)+(1-s)    \nabla V\left(p -\frac12 \eta \nabla U(q)\right)\right) \left[\nabla V(p)-  \nabla V\left(p -\frac12 \eta \nabla U(q)\right)\right]\right] \;ds\;dt
\end{align*}
\end{lem}

\begin{lem}
\label{lem:P_error_general}
Under Assumptions~\ref{asm:matrixspace} and~\ref{asm:third_derivative}, 
we have, 
\begin{align}
&|||P(\eta)- {\hat p}|||_2\nonumber \\\le&(L_V(d+2))^{1/2}\left\{\frac{T_U(L_U)^{3/2} (d+2)^{1/2}+(L_U)^{3/2}(L_V)^{1/2}}{6}+ \frac{T_U}{12}+\frac{(L_U)^{3/2}(L_V)^{1/2}}{4}\right\}\eta^3.\label{eqn:P_error_general}
\end{align}
\end{lem}
\begin{proof}
From above derivations, we know that, 
$|||P(\eta)- {\hat p}|||_2\le |||B_{11}|||_2+|||B_{121}|||_2+|||B_{122}|||_2+|||B_2|||_2$.Therefore,
\begin{align*}
|||B_{11}|||_2\le & \int_0^\eta \int_0^t |||\nabla^2 U(Q(s)) \nabla V(P(s)) -[\nabla^2 U(q+s \nabla V(p))] \nabla V(P(s)) |||_2 \;ds \;dt
\\ \le &\sup_{q \in \bbQ}\Big\|\nabla^3U(q)\Big\|\int_0^\eta \int_0^t \int_0^s \ex\left[\bigg\| \nabla V(P(\tau)) -\nabla V(p)\bigg\|_2\cdot \bigg\|\nabla V(P(s)) \bigg\|_2 \right]\;ds \;dt
\\ & \hbox{\quad assumption on the operator norm of the third degree tensor}\\ \le &\sup_{q \in \Real^d}\Big\|\nabla^3U(q)\Big\| \int_0^\eta \int_0^t \int_0^s \Bigg| \Bigg|\Bigg|\nabla V(P(\tau)) -\nabla V(p)\Bigg| \Bigg|\Bigg|_2\cdot \Bigg| \Bigg|\Bigg|\nabla V(P(s)) \Bigg| \Bigg|\Bigg|_2 \;ds \;dt
\\&\hbox{\quad Cauchy-Schwartz inequality}
\\\le  & \frac{\sup_{q \in \Real^d}\Big\|\nabla^3U(q)\Big\|(L_U)^{3/2}(L_V)^{1/2} (d+2)}{6}\eta^3.
\\ &  \hbox{\quad Lemma ~\ref{lem:moment_of_UV}}
\end{align*}
\begin{align*}
|||B_{121}|||_2\le &
\int_0^\eta \int_0^t  \Bigg| \Bigg|\Bigg|[\nabla^2 U(q+s \nabla V(p))] \nabla V(P(s)) -\int_0^t [\nabla^2 U(q+s \nabla V(p))] \nabla V(p)  \Bigg| \Bigg|\Bigg|_2 \;ds \;dt \\ \le & L_U \int_0^\eta \int_0^t  \Bigg| \Bigg|\Bigg| \nabla V(P(s)) - \nabla V(p)  \Bigg| \Bigg|\Bigg|_2 \;ds \;dt\\ & \hbox{\quad assumption on $U$}
\\ \le & L_U L_V\int_0^\eta \int_0^t  \Bigg| \Bigg|\Bigg| P(s)- p   \Bigg| \Bigg|\Bigg|_2 \;ds \;dt
\\ & \hbox{\quad assumption on $V$}\\
 \le & L_U L_V\int_0^\eta \int_0^t \int_0^s ||| \nabla U(\tau) |||_2 \:d\tau  \;ds \;dt\\
 &\hbox{\quad an application of Newton-Leibnitz}\\ \le &\frac{ (L_U)^{3/2}L_V(d+2)^{1/2} }{6}\eta^3
\end{align*}
Apply Lemma~\ref{lem:tech_one}, we have, 
\begin{align*}
|||B_{122}|||_2\le &
\Bigg| \Bigg|\Bigg|
\int_0^\eta \int_0^t [\nabla^2 U(q+s \nabla V(p))] \nabla V(p) \;ds \;dt-
\frac12\int_0^\eta \int_0^\eta  [\nabla^2 U(q+s \nabla V(p))] \nabla V(p) \;dt\Bigg| \Bigg|\Bigg|_2\\ 
\le & \frac{\sup_{q \in \Real^d}\Big\|\nabla^3U(q)\Big\|(L_V)^{1/2}(d+2)^{1/2}}{12}\eta^3 
\end{align*}
\begin{align*}
|||B_2|||_2\le &\frac12\int_0^\eta \int_0^\eta  \Bigg| \Bigg|\Bigg|\nabla^2 U\left(q+s\nabla V(p)+(1-s)    \nabla V\left(p -\frac12 \eta \nabla U(q)\right)\right) \left[\nabla V(p)- \nabla V\left(p -\frac12 \eta \nabla U(q)\right)\right]\Bigg| \Bigg|\Bigg|_2 \;ds\;dt\\
\le &\frac12\int_0^\eta \int_0^\eta L_U\Bigg| \Bigg|\Bigg| \left[\nabla V(p)- \nabla V\left(p -\frac12 \eta \nabla U(q)\right)\right]\Bigg| \Bigg|\Bigg|_2 \;ds\;dt
\\
\le &\frac{(L_U)^{3/2}L_V(d+2)^{1/2}}{4}\eta^3
\end{align*}
\end{proof}

The following technical lemma and its proof are included for completeness.
\begin{lem}
\label{lem:tech_one}
For a locally integrable function $f(\cdot)$, we have, 
\begin{align*}
\int_0^\eta \int_0^t f(s)\;ds \;dt - \frac12\int_0^\eta \int_0^\eta f(s)\;ds \;dt =&\int_0^\eta \frac{\tau}{2}(\tau-\eta)
f'(\tau)\;d\tau.
\end{align*}
\end{lem}
\begin{proof}
\begin{align*}
&\int_0^\eta \int_0^t f(s)\;ds \;dt - \frac12\int_0^\eta \int_0^\eta f(s)\;ds \;dt \\
=&\int_0^\eta \int_s^\eta f(s)\;dt \;ds -  \frac{\eta}{2}\int_0^\eta f(s)\;ds 
\\
=&\int_0^\eta f(s)(\eta-s) \;ds -  \frac{\eta}{2}\int_0^\eta f(s)\;ds 
\\
&= \int_0^\eta f(s)(\frac{\eta}{2}-s) \;ds \\
&= \int_0^\eta [f(s)-f(0)](\frac{\eta}{2}-s) \;ds
\\
&= \int_0^\eta \int_0^sf'(\tau)\;d\tau(\frac{\eta}{2}-s) \;ds
\\
&= \int_0^\eta \int_\tau^\eta
f'(\tau)(\frac{\eta}{2}-s) \;ds \;d\tau
\\
&=\int_0^\eta \frac{\tau}{2}(\tau-\eta)
f'(\tau) \;d\tau
\end{align*}

\end{proof}

\section{Lower Bounding the Acceptance Probability}
\label{sec:acceptance_probability}

From the expression of the acceptance/rejection probability calculation\eqref{eqn:acceptance_probability}, we know that it suffices to show that there exists $a>0$, such that,  $|U(q_K)+V(p_K)-U(q_0)-V(p_0)|\le a$. Meanwhile, since the Hamiltonian is invariant for the exact solution, hence this quantity become the difference between the exact Hamiltonian and that of the symplectic integrator. 
In essence, we need to estimate $|U({\hat q})-U(Q(\eta))|$ and $|V({\hat p})-V(P(\eta)|$.
Lemmata~\ref{lem:Q_error_general} and~\ref{lem:P_error_general} imply that these terms are  of order $\eta^3$.
Then the desired results follows.  This is a similar result to Chen \& Gatmiry where they say that, for general subexponential target probability distribution, there exists a compact set $\Lam \in \Real^d\times\Real^d$, such that when the initial point $(q_0,p_0)\in \Lam$, the acceptance can be bounded from below. 

\begin{align*}
\ex|U({\hat q})-U(Q(\eta))| \le & \ex\bigg|\int_0^1 \nabla U(q_s)\cdot[{\hat q}-Q(\eta)] ds \bigg| \\ \le &\ex\bigg|\int_0^1 [\nabla U(q_s)-\nabla U(q) ]\cdot[{\hat q}-Q(\eta)] ds \bigg|+\int_0^1 \ex\bigg|\nabla U(q) \cdot[{\hat q}-Q(\eta)] ds \bigg|
\\ \le &L_U\ex\bigg|\int_0^1 \int_0^s [q_\tau-q]\cdot[{\hat q}-Q(\eta)]\;d\tau\; ds \bigg|+\int_0^1 \ex\bigg|\nabla U(q) \cdot[{\hat q}-Q(\eta)] ds \bigg|
\\ \le & L_U \int_0^1 \int_0^s |||q_\tau-q|||_2|||[{\hat q}-Q(\eta)|||_2\;d\tau \;ds + |||\nabla U(q)|||_2|||[{\hat q}-Q(\eta)|||_2\\ \le& (L_U|||q|||_2+|||\nabla U(q)|||_2)|||)|||{\hat q}-Q(\eta)|||_2
\\ \le &[L_U\sigma_q+(dL_U)^{1/2}]\left\{\frac{T_V (d+2)L_U}{24}+\frac{L_VL_U [(d+2)L_V]^{\frac12}}{6}\right\}\eta^3,.
\end{align*}
with $q_s=(s{\hat q}+(1-s)Q(\eta))$. All the quantities are available through Lemmata~\ref{lem:moment_of_UV} and~\ref{lem:Q_error_general}. By the same arguments, we can obtain the following estimate for $\ex|V({\hat p})-V(P(\eta)|$ which can be bounded further with Lemmata~\ref{lem:moment_of_UV} and~\ref{lem:P_error_general}
\begin{align*}
&\ex|V({\hat p})-V(P(\eta))|\\
\le &
(L_V|||p|||_2+|||\nabla V(q)|||_2)|||)|||{\hat p}-P(\eta)|||_2 \\ 
\le & [L_V\sigma_p+(dL_U)^{1/2}](L_V(d+2))^{1/2}\left\{\frac{T_U(L_U)^{3/2} (d+2)^{1/2}+(L_U)^{3/2}(L_V)^{1/2}}{6}\right.
\\& \left.+ \frac{T_U}{12}+\frac{(L_U)^{3/2}(L_V)^{1/2}}{4}\right\}\eta^3.
\end{align*}

\section{Proof of Lemma~\ref{lem:displacement}}
\label{sec:proof_of_key_lemma}
\begin{proof}[Proof of Lemma~\ref{lem:displacement}]
We will first show that~\eqref{eqn:displacement} holds for Schwartz functions $h\in L^2(\Real^d, \bbf(q)dq)$ with uniform constants and $K=1$, then a mollification procedure, see e.g.~\cite{lieb2001analysis} establishes~\eqref{eqn:displacement} for all $h\in L^2(\Real^d, \bbf(q)dq)$. Repeat the arguments, ~\eqref{eqn:displacement} can be obtained for general $K$. First, 
\begin{align*}
&\int_{\Real^d} h(q)\int_{\Real^d} [h(q)-h({\hat q})]g(p) dp dq \\ =& \int_{\Real^d} h(q)\int_{\Real^d} \left[h(q)-h(q+\eta \nabla V(p-\frac{\eta}{2}\nabla U(q)))\right]g(p) dp dq
\\ =& \underbrace{\int_{\Real^d}h(q)\int_{\Real^d} [h(q)-h(q+\eta \nabla V(p))]g(p) dp dq}_{III_1}\\&+\underbrace{\int_{\Real^d} h(q)\int_{\Real^d}\left[ h(q+\eta \nabla V(p))-h(q+\eta \nabla V(p-\frac{\eta}{2}\nabla U(q)))\right]g(p) dp dq}_{III_2}.
\end{align*}
For $III_1$ and $III_2$
\begin{align*}
III_1=&\int_{\Real^d} h(q)\int_{\Real^d} [h(q)-h(q+\eta \nabla V(p))]g(p) dp dq\\ \stackrel{(1)}{=}& \underbrace{-\frac{\eta^2}{2}\int_{\Real^d}\nabla V(p) \cdot \left[\int_{\Real^d}h(q)\nabla^2h(q) f(q)  dq \right] \cdot \nabla V(p)  g(p)dp}_{III_{11}}\\ &+\underbrace{\frac{\eta^2}{2}\int_{\Real^d}\nabla V(p) \cdot \left[\int_{\Real^d}h(q)[\nabla^2h(q)-\nabla^2h({\tilde q})] f(q)  dq \right] \cdot \nabla V(p)  g(p)dp}_{III_{12}}.
\end{align*}
with (1) is due to the  zero mean assumption of $p$ and mean value theorem with ${\tilde q}$. For $III_{11}$, we have, 
\begin{align*}
III_{11}=& -\frac{\eta^2}{2}\int_{\Real^d}\nabla V(p) \cdot \left[\int_{\Real^d}h(q)\nabla^2h(q) f(q)  dq \right] \cdot \nabla V(p)  g(p)dp\\
=&\frac{\eta^2}{2}\int_{\Real^d}\nabla V(p) \cdot\int_{\Real^d} [\nabla h(q) \otimes \nabla h(q)  f(q)]+ h(q) \nabla h(q) \otimes \nabla f(q) ] dq \cdot \nabla V(p)  g(p)dp
\\
=&\frac{\eta^2\sigma_p^2}{2}\int_{\Real^d} \|\nabla h(q)\|^2_2 - \sum_{i,j} h(q) \partial_i h(q)  \partial_j U(q) ] f(q)dq \cdot \mu_{ij}, 
\end{align*}
by the definition of $\mu_{ij} =  \int_{\Real^d} \partial_i V(p) \partial_j V(p) g(p) dp$. For $III_{12}$, from a direct application of the H\'older's inequality uniform bounded on third order derivatives on the mollified function, see, e.g.~\cite{lieb2001analysis}. we know that, there exists a $C_2>0$, such that, 
\begin{align*}
\bigg\|\int_{\Real^d}h(q)\|\nabla^2h(q)-\nabla^2h({\tilde q})]\|_F f(q)  dq\bigg\| \le C_2\eta\|h\|_2^2.
\end{align*}
For $III_2$, we have, 
\begin{align*}
\int_{\Real^d}h(q) \nabla h(q) \cdot \nabla U(q) ] f(q)dq \cdot \int_{\Real^d}\nabla^2 V(p)  g(p)dp,
=\sum_{i,j} h(q) \partial_i h(q)  \partial_j U(q) ] f(q)dq \cdot \sigma_{ij}, 
\end{align*}
by the definition of $\sigma_{ij} =  \int_{\Real^d} \partial_{ij} V(p) g(p) dp$. Therefore, the cancellation follows from the assumption  $\sigma_{ij}=\mu_{ij}$.
Meanwhile 
\begin{align*}
III_2=&\int_{\Real^d} h(q)\int_{\Real^d}\left[h(q+\eta \nabla V(p))-h(q+\eta \nabla V(p-\frac{\eta}{2}\nabla U(q)))\right]g(p) dp dq \\=& \int_{\Real^d}h(q) \nabla h(q) \cdot \nabla U(q) ] f(q)dq \cdot \int_{\Real^d}\nabla^2 V(p)  g(p)dp,
\end{align*}
by the convexity assumption on $U(q)$.  Therefore, we have, $\int_{\Real^d} h(q)\int_{\Real^d} [h(q)-h({\hat q})]g(p) dp dq \ge \frac{\varrho\eta^2}{2} C_1\sigma_{V}^2\|h\|_2$ with $C_1$ being the optimal Poincar\'e inequality constant, and $\sigma_{V}^2:= \int_{\Real^d}\|\nabla V(p) \|_2 g(p)dp$.

Denote $A(q,p)$ as the event that the proposal is accepted. $\|h\|_2^2-\langle h, M_Hh\rangle = \int_{\Real^d} h(q)\int_{\Real^d} [h(q)-h({\hat q})]f(q)g(p) dp dq + \int_{\Real^d} h(q)\int_{\Real^d} [1-{\bf 1}_{A(q,p)}][h(q)-h({\hat q})]g(p) dp dq$. We can bound the second term with higher order of $A_3\eta^3\|h\|_2^2$ (at least $\eta^3$) using H\"older's inequality and Lemma~\ref{lem:acceptance_prob}. Therefore, we have, $\|h\|_2^2-\langle h, M_Hh\rangle \ge \eta^2(\frac{C_1\sigma_{V}^2}{2} -A_3\eta)\|h\|_2^2$.
\end{proof}

\section{Density of Pushforward Auxiliary Distributions}
\label{sec:continuity}

Fix $\bbq \in \Real^d$, the probability measure $\calP_q$ of the image $Q\in \Real^d$ can be viewed as a pushforward of the auxiliary probability measure via the integrator, therefore, its density bears the following form,
\begin{align}
\bbg (\bbp(\bbq, Q)) \det \left( \frac{\partial \bbp(\bbq, Q)}{\partial Q}\right),
\end{align}
with $\bbp(\bbq, Q)$ denotes the inverse of the integrator. 

\textbf{Kullback-Leibler(KL) divergence calculation}

For any pair $\bbq_1, \bbq_2\in \Real^d$, the Kullback-Leibler(KL) divergence $KL (\calP_{\bbq_1}|| \calP_{\bbq_2})$ can be written as, 
\begin{align*}
KL (\calP_{\bbq_1}|| \calP_{\bbq_2})= & \int_\Real^d \bbg (\bbp(\bbq_1, Q)) \det \left( \frac{\partial \bbp(\bbq_1, Q)}{\partial Q}\right)\log\left(\frac{\bbg (\bbp(\bbq_1, Q)) \det \left( \frac{\partial \bbp(\bbq_1, Q)}{\partial Q}\right)}{\bbg (\bbp(\bbq_2, Q)) \det \left( \frac{\partial \bbp(\bbq_2, Q)}{\partial Q}\right)}\right) dQ\\
\stackrel{(1)}{=}&\int_\bbP  \log\left(\frac{\bbg (\bbp) \det \left( \frac{\partial \bbp(\bbq_1, Q(\bbq_1, \bbp))}{\partial Q}\right)}{\bbg (\bbp(\bbq_2, Q(\bbq_1, \bbp))) \det \left( \frac{\partial \bbp(\bbq_2, Q(\bbq_1, \bbp)}{\partial Q}\right)}\right) \bbg (\bbp) d\bbp
\\ \stackrel{(2)}{=}&\int_\bbP  \log\left(\frac{\bbg (\bbp) }{\bbg (\bbp(\bbq_2, Q(\bbq_1, \bbp))) \det \left( \frac{\partial \bbp(\bbq_2, Q(\bbq_1, \bbp)}{\partial Q}\right)}\right) \bbg (\bbp) d\bbp\\
=&
\int_\bbP \left( \log\bbg (\bbp)-\log[\bbg (\bbp(\bbq_2, Q(\bbq_1, \bbp)))]-\log 
\det \left( \frac{\partial \bbp(\bbq_2, Q(\bbq_1, \bbp)}{\partial Q}\right)
\right) \bbg (\bbp) d\bbp,
\end{align*}
equation (1) is the result of change of variable from $Q$ to $p$, and (2) is due to the fact that,
\begin{align*}
\frac{\partial \bbp(\bbq_1, Q(\bbq_1, \bbp))}{\partial Q}=Id.
\end{align*}

Note that, conceptually, the term $\bbp(\bbq_2, Q(\bbq_1, \bbp))$ is treated as perturbation of $\bbp$, denoted as ${\tilde \bbp}=\bbp+\epsilon$, then it can be see that 
\begin{align*}
&\int_\bbP \left( \log\bbg (\bbp)-\log[\bbg (\bbp(\bbq_2, Q(\bbq_1, \bbp)))]-\log 
\det \left( \frac{\partial \bbp(\bbq_2, Q(\bbq_1, \bbp)}{\partial Q}\right)
\right) \bbg (\bbp) d\bbp\\
= & \int_\bbP \left( \log \left(\frac{\bbg (\bbp)}{\bbg ({\tilde \bbp})}\right) - \log 
\det \partial\bbg ({\tilde \bbp})\right)  \bbg (\bbp) d\bbp
\\ =& \int_\bbP \left( \log \left(1+\epsilon_1\right) - \log 
\det (I+\epsilon_2)\right)  \bbg (\bbp) d\bbp
\end{align*}

Now, let us examine them more carefully. Recall that $\bbg(\bbp) = \exp[-V(\bbp)]$, so, 
\begin{align}
\label{eqn:logtoV}
\log \left(\frac{\bbg (\bbp)}{\bbg ({\tilde \bbp})}\right)=& V({\tilde \bbp})- V(\bbp).
\end{align}
Let us see how ${\tilde \bbp}$ is calculated. We start with the case that only one leapfrog step is taken,
\begin{align*}
{\hat q}_1 =& q_1+\eta \nabla V\left(p -\frac12 \eta \nabla U(q_1)\right), \\
{\hat p}_1 =& p-\frac12 \eta \nabla U(q_1) - \frac12 \eta \nabla U({\hat q}_1).
\end{align*}
So ${\tilde \bbp}$ satisfies, 
\begin{align*}
q_2+\eta \nabla V\left({\tilde \bbp} -\frac12 \eta \nabla U(q_2)\right) =& q_1+\eta \nabla V\left(p -\frac12 \eta \nabla U(q_1)\right).
\end{align*}
Therefore, we can compute $\frac{\partial {\tilde \bbp}}{\partial \bbp}$,
\begin{align*}
\eta \nabla^2 V\left({\tilde \bbp} -\frac12 \eta \nabla U(q_2)\right)
\frac{\partial {\tilde \bbp}}{\partial \bbp}=&\eta \nabla^2 V\left(p -\frac12 \eta \nabla U(q_1)\right),
\end{align*}
Hence, 
\begin{align*}
\frac{\partial {\tilde \bbp}}{\partial \bbp}=&\left(\nabla^2 V\left({\tilde \bbp} -\frac12 \eta \nabla U(q_2)\right)\right)^{-1}\nabla^2 V\left(p -\frac12 \eta \nabla U(q_1)\right).
\end{align*}
Furthermore, we know that, 
\begin{align*}
\frac{\partial \bbp(\bbq_1, Q)}{\partial Q}=& \frac{\partial {\tilde \bbp}}{\partial \bbp}\cdot\frac{\partial  \bbp}{\partial Q}.
\end{align*}
These calculations will make an estimation of \eqref{eqn:logtoV} possible. More specifically, Hessian Lipschitz condition will lead to 
\begin{align*}
\frac{\partial {\tilde \bbp}}{\partial \bbp}=&\left(\nabla^2 V\left({\tilde \bbp} -\frac12 \eta \nabla U(q_2)\right)\right)^{-1}\nabla^2 V\left(p -\frac12 \eta \nabla U(q_1)\right)\\
=&I-\left(\nabla^2 V\left({\tilde \bbp} -\frac12 \eta \nabla U(q_2)\right)\right)^{-1}\left[\nabla^2 V\left({\tilde \bbp} -\frac12 \eta \nabla U(q_2)\right)-\nabla^2 V\left(p -\frac12 \eta \nabla U(q_1)\right)\right]
\end{align*}
We know that
\begin{align*}
\left(\nabla^2 V\left({\tilde \bbp} -\frac12 \eta \nabla U(q_2)\right)\right)^{-1}
\end{align*}
is bounded ($\succeq mI$ by strong convexity). Therefore, we only need to bound 
\begin{align*}
\Big\|
\nabla^2 V\left({\tilde \bbp} -\frac12 \eta \nabla U(q_2)\right)-\nabla^2 V\left(p -\frac12 \eta \nabla U(q_1)\right)\Big\|
\end{align*}
Hessian Lipschitz condition will lead to 
\begin{align}
&\Big\|
\nabla^2 V\left({\tilde \bbp} -\frac12 \eta \nabla U(q_2)\right)-\nabla^2 V\left(p -\frac12 \eta \nabla U(q_1)\right)\Big\|\nonumber \\ \le &\frac{\|\nabla^3V\|\eta}{2}\tvert\nabla U(q_1)- \nabla U(q_2)\tvert\le \frac{\eta\|\nabla^3V\|L_U}{2}\|q_1- q_2\| .\label{eqn:MO_continuity_I}
\end{align}

Inequality \eqref{eqn:MO_continuity_I}  thus give Lemmata~\ref{lem:continuity_general} and~\ref{lem:continuity_SG}, similar to Lemma 2 in~\cite{ChenGatmiry}.


\end{document}